\documentclass[a4paper,12pt,onecolumn]{article}
\usepackage{etex}

\usepackage[vmargin=2cm,hmargin=2cm,headheight=14.5pt,top=2cm,headsep=.5cm]{geometry}

\usepackage{bm}
\usepackage{empheq}
\usepackage{stackrel}
\usepackage{cases}
\usepackage{mathtools}
\usepackage{amsthm,amsmath,amscd}
\usepackage{makeidx}
\usepackage[all]{xy}
\usepackage{perpage}
\usepackage{tikz-cd}
\tikzcdset{every label/.append style = {font = \small}}
\usepackage[symbol]{footmisc}

\MakePerPage[2]{footnote}
\usepackage{graphicx}
\DeclareMathSizes{12}{12}{8}{6}

\usepackage{cite}
\usepackage{url}
\usepackage[charter]{mathdesign}
\usepackage{accents}
\usepackage{hyperref}

\newtheoremstyle{ptheorem}{1em}{0em}{\itshape}{}{\bfseries}{.}{.5em}{\thmname{#1}\thmnumber{ #2}\thmnote{ (\hspace{-.01pt}{#3})}}

\theoremstyle{ptheorem}

\newtheorem{thm}{Theorem}[section]
\newtheorem{pro}[thm]{Proposition}
\newtheorem{lem}[thm]{Lemma}
\newtheorem{cor}[thm]{Corollary}

\newtheoremstyle{hdef}{1em}{0em}{}{}{\bfseries}{.}{.5em}{\thmname{#1}\thmnumber{ #2}\thmnote{ (\hspace{-.01pt}{#3})}}
\theoremstyle{hdef}

\newtheorem{rem}[thm]{Remark}

\makeatletter
\newtheoremstyle{premark}{1em}{0em}{
}{}{\scshape}{.}{.5em}{}
\makeatother

\theoremstyle{premark}

\newtheorem{exa}[thm]{Example}

\numberwithin{equation}{section}
\numberwithin{figure}{section}

\DeclareMathOperator{\dif}{d}

\newcommand{\cC}{{\mathcal C}}

\newcommand{\bR}{{\mathbb R}}

\newcommand{\bZ}{{\mathbb Z}}

\renewcommand{\a}{\alpha}

\renewcommand{\l}{\lambda}

\renewcommand{\phi}{\varphi}

\newcommand{\ol}{\overline}
\newcommand{\fa}{\forall}

\newcommand{\nkp}{\enskip}

\newcommand{\sfa}{\nkp\fa}
\renewcommand{\d}{\delta}

\renewcommand{\(}{\left(}
\renewcommand{\)}{\right)}
\newcommand{\olb}[1]{%
	\vbox{\offinterlineskip\ialign{\hfil##\hfil\cr $\rotatebox[origin=c]{90}{$]$}$\cr\noalign{\kern-.45ex}{$#1$}\cr}}}
\allowdisplaybreaks
\begin{document}
	\title{Solutions of the first order linear equation with reflection and general linear conditions\footnote{Partially supported by  Ministerio de Econom\'ia y Competitividad (Spain) project MTM2013-43014-P and Xunta de Galicia (Spain), project EM2014/032.}}
	
	\author{
		Alberto Cabada and F. Adri\'an F. Tojo\footnote{Supported by  FPU scholarship, Ministerio de Educaci\'on, Cultura y Deporte (Spain).} \\
		\normalsize
		Departamento de An\'alise Ma\-te\-m\'a\-ti\-ca, Facultade de Matem\'aticas,\\ 
		\normalsize Universidade de Santiago de Com\-pos\-te\-la, Spain.\\ 
		\normalsize e-mail: \textit{alberto.cabada@usc.es, fernandoadrian.fernandez@usc.es}\\
	}
	\date{}
	
	\maketitle

\begin{abstract}
This work is devoted to the study of first order linear problems with involution and general linear conditions. We first study the problem in the case of antiperiodic boundary conditions, giving an explicit Green's function for it. Then we move forward to more general linear boundary conditions, focusing on sufficient conditions for existence and uniqueness of solution. At the end of the paper we give estimates that ensure the positivity of the solution in the general problems and illustrate these applications with examples.
\end{abstract}


\pagestyle{headings}
\section{Introduction}

The study of functional differential equations with involutions (DEI) can be traced back to the solution of the equation $x'(t)=x(1/t)$ by Silberstein (see \cite{Sil}) in 1940. Briefly speaking, an involution is just a function $f$ that satisfies $f(f(x))=x$ for every $x$ in its domain of definition. For most applications in analysis, the involution is defined on an interval of $\bR$ and, in the majority of the cases, it is continuous, which implies it is decreasing and has a unique fixed point \cite{Wie}. Ever since that foundational paper of Siberstein, the study of problems with DEI has been mainly focused on those cases with initial conditions, with an extensive research in the case of the reflection $f(x)=-x$.\par
Wiener and Watkins study in \cite{Wie} the solution of the equation $x'(t)-a\, x(-t)=0$ with initial conditions. Equation $x'(t)+a\, x(t)+b\,x(-t)=g(t)$ has been treated by Piao in \cite{Pia, Pia2}. In \cite{Kul, Sha, Wie, Wat1, Wie2} some results are introduced to transform this kind of problems with involutions and initial conditions into second order ordinary differential equations with initial conditions or first order two dimensional systems, granting that the solution of the last will be a solution to the first. Furthermore, asymptotic properties and boundedness of the solutions of initial first order problems are studied in \cite{Wat2} and \cite{Aft} respectively. Second order boundary value problems have been considered in \cite{Gup, Gup2, Ore2, Wie2} for Dirichlet and Sturm-Liouville boundary value conditions and higher order equations has been studied in \cite{Ore}. Other techniques applied to problems with reflection of the argument can be found in \cite{And, Ma, Wie1}.\par
More recently, the papers of Cabada et al. \cite{Cab4, Cab5} have further studied the case of the first order linear equation with two-point boundary conditions, adding a new element to the previous studies: the existence of a Green's function. Once the study of the sign of the aforementioned function is done, maximum and anti-maximum principles trivially follow.\par
In this paper we try to extend the results in \cite{Cab4} to a more general kinds of boundary conditions for the equation $x'(t)+m\,x(-t)=h(t)$.\par
The structure of this paper is as follows. In Section 1 we will study the antiperiodic case. As it will be shown, the properties and results for this case run parallel to the peridic case, so the results are much similar to the ones found in \cite{Cab4}.\par
In section 2 we go an step further and try to solve the equation for conditions given by a linear functional acting on the solution. Then, in Section 3, we apply our knowledge of the Green's function of the periodic case in order to derive new positivity results. Finally, these results are illustrated with examples.

\subsection{The antiperiodic case}

In \cite{Cab4,Cab5} the following periodic problem was studied
\begin{equation}\label{eq1} x'(t)+m\,x(-t)=h(t),\ t\in I;\quad x(-T)=x(T).\end{equation}
where $m\in\bR$, $T\in\bR^+$, $h\in L^1(I)$ and $I:=[-T,T]$.\par
The main result of \cite{Cab4} was the following.
\begin{pro}\label{Greenf} Suppose that $m \neq k \, \pi/T$, $k \in \bZ$. Then problem (\ref{eq1}) has a unique solution given by the expression
\begin{equation}
\label{e-u}
u(t):=\int_{-T}^T\ol G(t,s)h(s)\dif s,
\end{equation}
where $$\ol{G}(t,s):=m\,G(t,-s)-\frac{\partial G}{\partial s}(t,s),$$
is called the \textbf{Green's function} related to problem (\ref{eq1}).
\end{pro}
The expression of $\ol G$ was found in \cite{Cab4} and, expressed in a simpler form in \cite{Cab5} with the change of variables $t=z\,T$, $s=y\, T$  and $\a=m\,T$.
\begin{equation}\label{eqgb2}
\sin(\alpha) \ol G(z,y)=
\begin{cases}
\cos[\alpha(1-z)-\frac{\pi}{4}]\cos(\alpha y-\frac{\pi}{4})& \text{if}\quad z>|y|, \\
\cos(\alpha z+\frac{\pi}{4})\cos[\alpha(y-1)-\frac{\pi}{4}] & \text{if}\quad |z|<y, \\
\cos(\alpha z+\frac{\pi}{4})\cos[\alpha(1+y)-\frac{\pi}{4}]& \text{if}\quad -|z|>y, \\
\cos[\alpha(z+1)+\frac{\pi}{4}]\cos(\alpha y-\frac{\pi}{4})& \text{if}\quad z<-|y|,
\end{cases}\quad z,y\in[-1,1]
\end{equation}
On the light of these results, in this section we try to solve the analogous problem
$$x'(t)+m\,x(t)=h(t),\quad x(-T)+x(-T)=0.$$
The development of this section is parallel to the results in \cite{Cab4} for, instead of studying the periodic case, we study the antiperiodic one. The proofs of the following results are analogous to the ones in \cite{Cab4}, so we omit them here.\par
If we consider the antiperiodic problem
\begin{equation}\label{anti} x'(t)+m\,x(-t)=h(t),\quad x(-T)+x(T)=0,\end{equation}
we have that the reduced problem for $h\equiv 0$ corresponds with the harmonic oscillator with antiperiodic boundary value conditions
\begin{equation*}\label{oa-anti} x''(t)+m^2\,x(t)=0,\quad x(-T)+x(T)=0,\quad x'(-T)+x'(T)=0\end{equation*}
of which the Green's function, $H$, is given by the expression
$$2m\cos(m\,T)H(t,s)=\begin{cases} \sin m(t-s-T) & \text{ if } -T\le s\le t\le T, \\ \sin m(s-t-T) & \text{ if } -T\le t<s\le T.\end{cases}$$\par
It is straight forward to check that the following properties are fulfilled.\par
$(A_1)$ $H\in\cC(I^2,\bR)$.\par
$(A_2)$ $\frac{\partial H}{\partial t}$ y $\frac{\partial^2 H}{\partial t^2}$ exist and are continuous on $I^2\backslash D$ where $D:=\{(t,s)\in I^2\ :\ t=s\}$.\par
Also,
$$2m\cos(m\,T)\frac{\partial H}{\partial t}(t,s)=\begin{cases} m\cos m(t-s-T) & \text{ if } -T\le s\le t\le T, \\ -m\cos m(s-t-T) & \text{ if } -T\le t<s\le T,\end{cases}$$
$$\lim_{s\to t^-}2m\cos(m\,T)\frac{\partial H}{\partial t}(t,s)=m\cos m\,T,$$
$$\lim_{s\to t^+}2m\cos(m\,T)\frac{\partial H}{\partial t}(t,s)=-m\cos m\,T,$$
hence\par
$(A_3)$ $\frac{\partial H}{\partial t}(t,t^-)-\frac{\partial H}{\partial t}(t,t^+)=1\sfa t\in I$.\par
Furthermore, we have the following \par
$(A_4)$ $\frac{\partial^2 H}{\partial t^2}(t,s)+m^2H(t,s)=0\sfa(t,s)\in I^2\backslash D$.\par
$(A_5)$\quad \begin{minipage}[t]{0.8\linewidth}\begin{itemize} \item[$a)$]  $H(T,s)+H(-T,s)=0\sfa s\in I,$
\item[$b)$] $\frac{\partial H}{\partial t}(T,s)+\frac{\partial H}{\partial t}(-T,s)=0\sfa s\in I.$
\end{itemize}\end{minipage}\par
For every $t,\,s\in I$, we have that \par
$(A_6)$ $H(t,s)=H(s,t).$\par
$(A_7)$ $H(t,s)=H(-t,-s).$\par
$(A_8)$ $\frac{\partial H}{\partial t}(t,s)=\frac{\partial H}{\partial s}(s,t).$\par
$(A_9)$ $\frac{\partial H}{\partial t}(t,s)=-\frac{\partial H}{\partial t}(-t,-s).$\par
$(A_{10})$ $\frac{\partial H}{\partial t}(t,s)=-\frac{\partial H}{\partial s}(t,s).$\par
The properties $(A_1)-(A_{10})$ are equivalent to the properties $(I)-(X)$ in \cite{Cab4}. This allows us to prove the following proposition analogously to \cite[Proposition 3.2]{Cab4}.
\begin{pro}\label{Greenfa} Assume $m \neq (k+\frac{1}{2}) \, \frac{\pi}{T}$, $k \in \bZ$. Then problem (\ref{anti}) has a unique solution
\begin{equation*}
u(t):=\int_{-T}^T\ol H(t,s)h(s)\dif s,
\end{equation*}
where $$\ol{H}(t,s):=m\,H(t,-s)-\frac{\partial H}{\partial s}(t,s)$$
is the Green's function \index{función de Green} relative to problem (\ref{anti}).
\end{pro}
The Green's function $\ol H$ has the following explicit expression:
\begin{equation*}\label{hbarra}
2\cos(mT)\ol H(t,s)=\begin{cases} \sin m(-T+s+t)+\cos m(-T-s+t) & \text{si}\quad t>|s|,\\\sin m(-T+s+t)-\cos m(-T+s-t) & \text{si}\quad |t|<s,\\\sin m(-T-s-t)+\cos m(-T-s+t) & \text{si}\quad -|t|>s,\\\sin m(-T-s-t)-\cos m(-T+s-t) & \text{si}\quad t<-|s|.\end{cases}
\end{equation*}
The following properties of $\ol H$ hold and are equivalent to properties $(I')-(V')$ of the Green's function $\ol G$ of the periodic problem in \cite{Cab4}.\par
$(A'_1)$ $\frac{\partial \ol H}{\partial t}$ exists and is continuous on $I^2\backslash D$,\par
$(A'_2)$ $\ol H(t,t^-)$ y $\ol H(t,t^+)$ exist for all $t\in I$ and satisfy $\ol H(t,t^-)-\ol H(t,t^+)=1\sfa t\in I$,\par
$(A'_3)$ $\frac{\partial\ol H}{\partial t}(t,s)+m\,\ol H(-t,s)=0$ a.\,e. $t,s\in I,\ s\ne t$,\par
$(A'_4)$ $\ol H(T,s)+\ol H(-T,s)=0\sfa s\in (-T,T)$,\par
$(A'_5)$ $\ol H(t,s)=\ol H(-s,-t)\sfa t,s\in I$.\par

Despite the parallelism with the periodic problem, we cannot generalize the maximum and anti-maximum results of \cite{Cab4} because property $(A'_4)$ guarantees that $\ol H(\cdot,s)$ changes sign for a.\,e. $s$ and, by property $(A'_5)$, that $\ol H(t,\cdot)$ changes sign for a.e. $t$ fixed.

\subsection{The general case}

In this section we study equation $x'(t)+m\,x(t)=h(t)$ under the conditions imposed by a linear functional $F$, this is, we study the problem
\begin{equation}\label{progen} x'(t)+m\,x(-t)=h(t),\quad F(x)=c,\end{equation}
where $c\in\bR$ and $F\in W^{1,1}(I)'$.\par
Remember that that $W^{1,1}(I):=\{f:I\to\bR\ :\ f'\in L^1\}$ and we denote by $W^{1,1}(I)'$ its dual. Also, we will denote by $\cC_c(I)$ the space of compactly supported functions on $I$.\par
Recall that the solutions of equation $x''(t)+m^2 x(t)=0$ are parametrized by two real numbers $a$ and $b$ in the following way: $u(t)=a\cos m\,t+b\sin m\,t$. Since every solution of equation $x'(t)+m\,x(-t)=0$ has to be of this form, if we impose the equation to be satisfied, we obtain a relationship between the parameters: $b=-a$, and hence the solutions of $x'(t)+m\,x(-t)=0$ are given by $u(t)=a(\cos m\,t-\sin m\,t)$, $a\in\bR$.\par
We now recall the following result from \cite{Cab4}.
\begin{lem}[{\cite[Corollary 3.4]{Cab4}}]\label{corlambda}
Suppose that $m \neq k \, \pi/T$, $k \in \bZ$. Then the problem 
\begin{equation*}\label{prolambda}
 x'(t)+m\, x(-t)=h(t),\nkp t\in I:=[-T,T],\quad  x(-T)-x(T)=\l,
\end{equation*}
with $\l\in\bR$ has a unique solution given by the expression
\begin{equation}\label{solgen}
u(t):=\int_{-T}^T\ol G(t,s)h(s)\dif s+\l\ol G(t,-T).
\end{equation}
\end{lem}
Observe that $2\sin m\,T\,\ol G(t,-T)=\cos m\,t-\sin m\,t$, and $\ol G(t,-T)$ is the unique solution of the problem
$$x'(t)+m\,x(-t)=0,\quad x(-T)-x(T)=1,$$  therefore the previous expression has the form of a particular solution --that of the periodic problem-- plus a general solution of the equation $x'(t)+m\,x(-t)=0$.
Hence, if impose the condition $F(x)=c$ on equation \eqref{solgen}, we have that
$$c=F\(\int_{-T}^T\ol G(t,s)h(s)\dif s\)+\l F(\ol G(t,-T))$$
and hence, for
$$\l=\frac{c-F\(\int_{-T}^T\ol G(t,s)h(s)\dif s\)}{F(\ol G(t,-T))},$$
expression \eqref{solgen} is a solution of problem \eqref{progen} as long as $F(\ol G(t,-T))\ne0$ or, which is the same,
$$F(\cos m\,t)\ne F(\sin m\,t).$$\par
We summarize this argument in the following result.
\begin{cor}
Assume $m \neq k \, \pi/T$, $k \in \bZ$, $F\in W^{1,1}(I)'$ such that $F(\cos m\,t)\ne F(\sin m\,t)$. Then problem \eqref{progen} has a unique solution given by
\begin{equation}\label{eqsolgen}u(t):=\int_{-T}^T\ol G(t,s)h(s)\dif s+\frac{c-F\(\int_{-T}^T\ol G(t,s)h(s)\dif s\)}{F(\ol G(t,-T))}\ol G(t,-T),\quad t\in I.\end{equation}
\end{cor}
\begin{rem} The condition  $m \neq k \, \pi/T$, $k \in \bZ$ together with the rest of the hypothesis of the corollary is sufficient for the existence of a unique solution of problem \eqref{progen} but is not necessary, as it has been illustrated in Proposition \ref{Greenfa}, because such a condition is only necessary for the existence of $\ol G$.
\end{rem}
\section{Applications}
We now apply the previous results in order to get some specific applications.
\begin{exa} Let $F\in W^{1,1}(I)'\cap\cC_c(I)'$ and assume $F(\cos m\,t)\ne F(\sin m\,t)$. The Riesz Representation Theorem guarantees the existence of a --probably signed-- regular Borel measure of bounded variation $\mu$ on $I$ such that $F(x):=\int_{-T}^{T}x\dif \mu$ and $\|F\|_{\cC_c(I)'}=|\mu|(I)$, where $|\mu|(I)$ is the total variation of the measure $\mu$ on $I$.\par
Let us compute now an estimate for the value of the solution $u$ at $t$.
\begin{align*}
|u(t)| & =\left|\int_{-T}^T\ol G(t,s)h(s)\dif s+\frac{c-F\(\int_{-T}^T\ol G(t,s)h(s)\dif s\)}{F(\ol G(t,-T))}\ol G(t,-T)\right|\\ & \le\sup_{s\in I}|\ol G(t,s)|\|h\|_1+\frac{|c-\int_{-T}^T\int_{-T}^T\ol G(t,s)h(s)\dif s\dif \mu(t)|}{|F(\ol G(t,-T))|}|\ol G(t,-T)|\\ &\le
\sup_{s\in I}|\ol G(t,s)|\|h\|_1+\frac{|c|+\sup_{t,s\in I}|\ol G(t,s)||\mu|(I)\|h\|_1}{|F(\ol G(t,-T))|}|\ol G(t,-T)|
\\ & =\left|\frac{c\,\ol G(t,-T)}{F(\ol G(t,-T))}\right|+\left[\sup_{s\in I}|\ol G(t,s)|+\left|\frac{\ol G(t,-T)}{F(\ol G(t,-T))}\right|\sup_{t,s\in I}|\ol G(t,s)||\mu|(I)\right]\|h\|_1.
\end{align*}
\end{exa}
Define operator $\Xi$  as $\Xi(f)(t):=\int_{-T}^T\ol G(t,s)f(s)\dif s$. And let us consider, for notational purposes, $\Xi(\d_{-T})(t):=\ol G(t,-T)$. Hence, equation \eqref{eqsolgen} can be rewritten as
\begin{equation}\label{solpos}u(t)=\Xi(h)(t)+\frac{c-F\(\Xi(h)\)}{F(\Xi(\d_{-T}))}\Xi(\d_{-T})(t),\quad t\in I.\end{equation}
We now recall a Lemma from \cite{Cab5}.
\begin{lem}[{\cite[Lemma 5.5]{Cab5}}]\label{midpoint} Let $f:[p-c,p+c]\to\bR$ be a symmetric function with respect to $p$, decreasing in $[p,p+c]$. Let $g:[a,b]\to\bR$ be an affine function such that $g([a,b])\subset[p-c,p+c]$. Under these hypothesis, the following hold.
\begin{itemize}
\item If $g(a)<g(b)<p$ or $p<g(b)<g(a)$ then $f(g(a))<f(g(b))$,
\item if $g(b)<g(a)<p$ or $p<g(a)<g(b)$ then $f(g(a))>f(g(b))$,
\item if $g(a)<p<g(b)$ then $f(g(a))<f(g(b))$ if and only if $g(\frac{a+b}{2})<p$,
\item if $g(b)<p<g(a)$ then $f(g(a))<f(g(b))$ if and only if $g(\frac{a+b}{2})>p$.
\end{itemize}
\end{lem}
With this Lemma, we can prove the following proposition.
\begin{pro}
Assume $\a\in(0,\pi/4)$, $F\in W^{1,1}(I)'\cap\cC_c(I)'$ such that $\mu$ is its associated Borel measure and $F(\cos m t)>F(\sin m t)$. Then the solution to problem \eqref{progen} is positive if
\begin{equation}\label{eq1sts}c>\frac{2M|\mu|(I)\|h\|_1}{1-\tan\a}.\end{equation}
\end{pro}
\begin{proof}
Observe that $\Xi(\d_{-T})(t)>0\sfa t\in I$ for every $\a\in(0,\frac{\pi}{4})$. Hence, if we assume that $u$ is positive,  solving for $c$ in \eqref{solpos}, we have that
$$c>F(\Xi(h))-F(\Xi(\d_{-T}))\frac{\Xi(h)(t)}{\Xi(\d_{-T})(t)}\sfa t\in I.$$
Reciprocally, if this inequality is satisfied, $u$ is positive.

It is easy to check using Lemma \ref{midpoint}, that
$$\min_{t\in I}\ol G(t,-T)=\frac{1}{2}(\cot\a-1)\quad\text{and}\quad\max_{t\in I}\ol G(t,-T)=\frac{1}{2}(\cot\a+1).$$ \par Let $M:=\max_{t,s\in I}\ol G(t,s).$\par Then
\begin{align*} & F(\Xi(h))-F(\Xi(\d_{-T}))\frac{\Xi(h)(t)}{\Xi(\d_{-T})(t)}\le |F(\Xi(h))|+\left|2F(\Xi(\d_{-T}))\frac{\Xi(h)(t)}{\cot\a-1}\right|\\ \le & M|\mu|(I)\|h\|_1+(\cot\a+1)|\mu|(I)\frac{M\|h\|_1}{\cot\a-1}=\frac{2M|\mu|(I)\|h\|_1}{1-\tan\a}.\end{align*}
Thus, a sufficient condition for $u$ to be positive is
$$c>\frac{2M|\mu|(I)\|h\|_1}{1-\tan\a}=:k_1.$$\par
\end{proof}
Condition \eqref{eq1sts} can be excessively strong in some cases, which can be illustrated with the following example.
\begin{exa} Let us assume that $F(x)=\int_{-T}^{T}x(t)\dif t$. For this functional,
$$\frac{2M|\mu|(I)\|h\|_1}{1-\tan\a}=\frac{4MT\|h\|_1}{1-\tan\a}.$$\par
In \cite[Lemma 5.11]{Cab5}, it is proven that $\int_{-T}^{T}\ol G(t,s)\dif t=\frac{1}{m}.$
Hence, we have the following
\begin{align*} & F(\Xi(h))-F(\Xi(\d_{-T}))\frac{\Xi(h)(t)}{\Xi(\d_{-T})(t)}=\int_{-T}^{T}\int_{-T}^{T}\ol G(t,s)h(s)\dif s\dif t-\int_{-T}^{T}\ol G(t,-T)\dif t\frac{\int_{-T}^{T}\ol G(t,s)h(s)\dif s}{\ol G(t,-T)}
\\= & \frac{1}{m}\int_{-T}^{T}h(s)\dif s-\frac{1}{m}\frac{\int_{-T}^{T}\ol G(t,s)h(s)\dif s}{\ol G(t,-T)}\le\frac{1}{m}\int_{-T}^{T}|h(s)|\dif s+\frac{1}{m}\frac{\int_{-T}^{T}\ol G(t,s)|h(s)|\dif s}{\ol G(t,-T)}\\ \le & \(1+\max_{t\in I}\frac{\max_{s\in I}\ol G(t,s)}{G(t,-T)}\)\frac{\|h\|_1}{m}\le\(1+\frac{M}{\min_{t\in I}G(t,-T)}\)\frac{\|h\|_1}{m}=\(1+\frac{2M}{\cot\a-1}\)\frac{\|h\|_1}{m}.\end{align*}
This provides a new sufficient condition to ensure that $u>0$.
$$c>\(1+\frac{2M}{\cot\a-1}\)\frac{\|h\|_1}{m}=:k_2.$$\par

Observe that
$$\frac{k_2}{k_1}=\frac{1+(2M-1)\tan\a}{4M\a}.$$
In order to quantify the improvement of the estimate, we have to know the value of $M$.
\begin{lem} $M=\frac{1}{2}(1 + \csc\a)$.
\end{lem}
\begin{proof}
By \cite[Lemma 5.9]{Cab5} we know that, after the change of variable $t=Tz$, $y=Ts$,
$$(\sin\a)\Phi(y)=\max_{z\in[-1,1]}\ol G(z,y)=\begin{cases}
\cos\left[\alpha(y-1)+\frac{\pi}{4}\right]\cos\(\alpha y-\frac{\pi}{4}\) & \text{if}\quad y\in\left[0,1\right], \\
\cos\(\alpha y+\frac{\pi}{4}\)\cos\left[\alpha(y+1)-\frac{\pi}{4}\right] & \text{if}\quad y\in[-1,0). \\
\end{cases}
$$
\par
Observe that $\Phi$ is symmetric, hence, it is enough to study it on $[0,1]$. Differentiating and equalizing to zero it is easy to check that the maximum is reached at $z=\frac{1}{2}$.
\end{proof}
Thus,
$$f(\a):=\frac{k_2}{k_1}=\frac{1}{2\a}\cdot\frac{1+\sec\a}{1+\csc\a}.$$
$f$ is strictly decreasing on $\(0,\frac{\pi}{4}\)$, $f(0^+)=1$ and $f\(\frac{\pi}{4}^-\)=\frac{2}{\pi}$.
\begin{figure}[hht]
\center{\includegraphics[width=.5\textwidth]{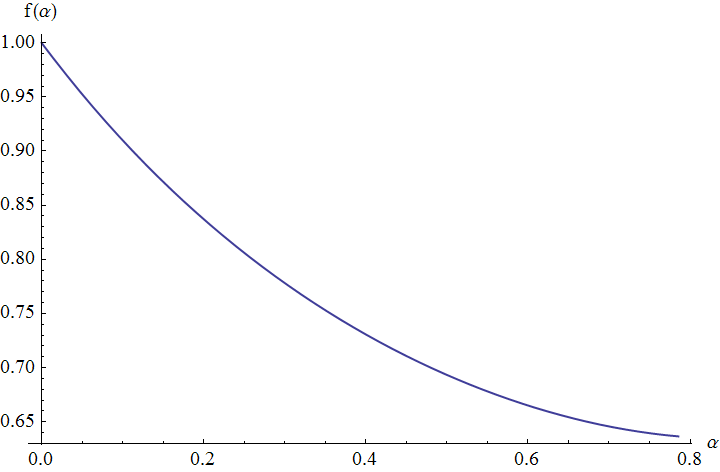}}\caption{$\frac{k_2}{k_1}$ as a function of  $\a$.}
\end{figure}\par
\end{exa}
\begin{figure}[b]
\center{\includegraphics[width=.5\textwidth]{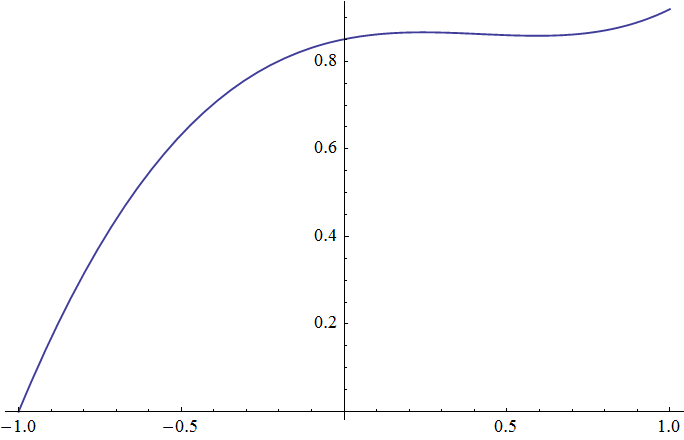}}\caption{Solution of problem \eqref{eqej1} for $c=0.850502\dots$}
\end{figure}\par
\begin{exa}We give now an example for which we compute the optimal constant $c$ that ensures the solution is positive and compare it to the aforementioned estimate. Consider the problem
\begin{equation}\label{eqej1} x'(t)+\,x(-t)=e^t,\ t\in\left[-\frac{1}{2},\frac{1}{2}\right],\quad \int_{-\frac{1}{2}}^\frac{1}{2}x(s)\dif s=c.\end{equation}
\end{exa}

For this specific case,
$$k_2=\frac{\cos\a+1}{\cos\a-\sin\a}\frac{\|h\|_1}{m}=\frac{2\cot\frac{1}{4}\sinh\frac{1}{2}}{\cot\frac{1}{4}-1}=4.91464\dots$$

Now, using the expression of $\ol G$, it is clear that
$$u(t)=\sinh t+\frac{c}{2\sin\frac{1}{2}}(\cos t-\sin t)$$
is the unique solution of problem \eqref{eqej1}. It is easy to check that the minimum of the solution is reached at $-1$ for $c\in[0,1]$. Also that the solution is positive for $c>2\sin\frac{1}{2}\sinh 1/(\cos 1+\sin 1)=0.850502\dots$, which illustrates that the estimate is far from being optimal.

\end{document}